\documentclass[11pt]{amsart}
\usepackage[english]{babel}
\usepackage[applemac]{inputenc}
\usepackage[T1]{fontenc}
\usepackage{palatino}
\usepackage{bbm}
\usepackage{float, verbatim}
\usepackage{amsmath}
\usepackage{amssymb}
\usepackage{amsthm}
\usepackage{amsfonts}
\usepackage{graphicx}
\usepackage{mathtools}
\usepackage{enumitem}
\usepackage[colorlinks = true, citecolor = black]{hyperref}
\usepackage{color}
\usepackage{tikz}
\usepackage{xcolor}
\usepackage{pgfplots}
\usepackage{caption}
\usepackage{subcaption}
\usepackage{enumerate}
\usepackage{url}

\newcommand{\boxd}{\dim_{\mathrm{B}}}

\newcommand{\Haus}{\dim_{\mathrm{H}}}
\newcommand{\Log}{\mathrm{Log}}

\newtheorem*{thm*}{Theorem}
\newtheorem*{conj*}{Conjecture}

\newtheorem{thm}{Theorem}[section]

\newtheorem{lma}[thm]{Lemma}
\newtheorem{cor}[thm]{Corollary}

\newtheorem{conj}[thm]{Conjecture}
\newtheorem{rem}[thm]{Remark}
\newtheorem{ques}[thm]{Question}

\begin{document}
	\title[Projections and number theory]{Fractal projections with an application in number theory}
	
	\author{Han Yu}
	\address{Han Yu\\
		Department of Pure Mathematics and Mathematical Statistics\\University of Cambridge\\CB3 0WB \\ UK }
	\curraddr{}
	\email{hy351@maths.cam.ac.uk}
	\thanks{}
	
	
	
	\maketitle
	
	\begin{abstract}
		In this paper, we discuss a connection between geometric measure theory and number theory. This method brings a new point of view for some number-theoretic problems concerning digit expansions.  Among other results, we showed that for each integer $k,$ there is a number $M>0$ such that if $b_1,\dots,b_k$ are multiplicatively independent integers greater than $M$, there are infinitely many integers whose base $b_1,b_2,\dots,b_k$ expansions all do not have any zero digits.
	\end{abstract}
	
	\maketitle
	\allowdisplaybreaks
	
	\section{Prime factors of binomial coefficients: Graham's question}
	In the 1970s, Erd\H{o}s, Graham, Ruzsa and Straus proved that there are infinitely many integers $n$ such that $\binom{2n}{n}$ is coprime with $3\times 5=15,$ see \cite{EGRS75}. Later on, Graham asked the following question.
	\begin{ques}[Graham's binomial coefficients problem]\label{QG}
		Are there infinitely many integers $n\geq 1$ such that the binomial coefficient $\binom{2n}{n}$ is coprime with $105=3\times 5\times 7?$
	\end{ques}
	\begin{rem}
		According to \cite{OEIS7}, Graham offers $1000\$$ to the first person with a solution.
	\end{rem}
	This problem turns out to be related with digit expansions of numbers in different bases. To be precise, let $b_1,\dots,b_k\geq 2$ be $k\geq 2$ different integers. For each $i\in\{1,\dots,k\},$ let $B_i\subset \{0,\dots,b_i-1\}$ be a subset of digits in base $b_i.$ Let $n,b\geq 2$ be integers, we write $D_b(n)$ for the set of digits used in representing $n$ in base $b.$ We define the  set of integers
	\[
	N_{b_1,\dots,b_k}^{B_1,\dots,B_k}=\{n\in\mathbb{N}: \forall i\in\{1,\dots,k\}, D_{b_i}(n)\subset B_i     \}.
	\]
	Thus $N_{b_1,\dots,b_k}^{B_1,\dots,B_k}$ contains integers with very special digit expansions simultaneously in many different bases. We will call such numbers to be with restricted digits. The original motivation of this type of problem is to study prime factors of $\binom{2n}{n}.$ The connection between prime factors of $\binom{2n}{n}$ and digits expansions of $n$ was established by Kummer in \cite{K52}.
	\begin{thm}[Kummer]\label{thm: Kummer}
		Let $p$ be a prime number. Then $p \not| \binom{2n}{n}$ if and only if the $p$-ary expansion of $n$ contains only digits $\leq (p-1)/2.$
	\end{thm}
	
	Due to Kummer's theorem, we see that for Graham's question, one needs to study the set $N_{3,5,7}^{B_3,B_5,B_7}$ where
	\[
	B_p=\{0,\dots,(p-1)/2\}.
	\]
	It is precisely the set of integers $n$ with $\binom{2n}{n}$ being coprime with $3,5,7.$ Graham's question is open, but there is some progress. In \cite{BY19}, it was proved that, under Schanuel's conjecture (see Conjecture \ref{schanuel} below and \cite{A71}), 
	\[
	\#N_{3,5,7}^{B_3,B_5,B_7}\cap [1,N]\leq N^{0.026}
	\]
	for all sufficiently large $N$. So we can say that there are not `too many' integers $n$ such that $\binom{2n}{n}$ is coprime with $3,5,7.$ Unconditionally, at least one of the results listed below holds for all large enough $N$:
	\[
	\#N_{3,5,7}^{B_3,B_5,B_7}\cap [1,N]\leq N^{0.026},
	\]
	\[
	\#N_{3,5,11}^{B_3,B_5,B_{11}}\cap [1,N]\leq N^{0.061},
	\]
	\[
	\#N_{3,5,13}^{B_3,B_5,B_{13}}\cap [1,N]\leq N^{0.073}.
	\]
	The results in \cite{BY19} suggest a more quantitative conjecture as follows.
	\begin{conj}\label{fu}
		Let $p_1,\dots,p_k$ be $k\geq 2$ different odd prime numbers. For each $i\in\{1,\dots,k\},$ let \[
		B_{i}=\{0,\dots,(p_i-1)/2\}.
		\]
		Consider the number
		\[
		s=\sum_{i=1}^k \frac{\log \#B_i}{\log p_i}=\sum_{i=1}^k \frac{\log (p_i+1)-\log 2}{\log p_i}.
		\]
		If $s\in (k-1,k),$ then for each $\epsilon>0$ there is a constant $C_\epsilon>1$ such that
		\begin{align}\label{*}
		C^{-1}_\epsilon N^{s-(k-1)-\epsilon}\leq \# N_{p_1,\dots,p_k}^{B_1,\dots,B_k}\cap [1,N]\leq C_\epsilon N^{s-(k-1)+\epsilon}
		\end{align}
		for all integers $N\geq 2.$ If $s<k-1$ then $N_{p_1,\dots,p_k}^{B_1,\dots,B_k}$ is finite.
	\end{conj}
	\begin{rem}
		The rightmost inequality of $(\ref{*})$ was proved in \cite{BY19} under Schanuel's conjecture. Thus the open problem is the leftmost inequality of $(\ref{*})$ and the finiteness statement. This is closely related to a Furstenberg's problem, see \cite{Fu2}, \cite{Wu} and \cite{Sh}.
	\end{rem}
	
	\section{Results in this paper}
	In this paper, we consider Graham's question and other problems related to digit expansions of numbers in different bases. We relate it to projections of fractal sets, a well-studied topic in geometric measure theory. This connection does not directly provide us with new results. It merely represents some number-theoretic problems using fractals (self-similar sets) and translates the problems into geometric properties (intersections, slices, projections) of those fractals. We will provide a more detailed discussion on this topic in Sections \ref{GMT} and \ref{Projection}. Here, we only need to know that $\Pi_x$ for $x\in\mathbb{R}^d$ stands for the map
	\[
	y\in\mathbb{R}^d\setminus \{x\}\to \Pi_x(y)=\frac{x-y}{|x-y|}\in S^{d-1}.
	\]
	Intuitively speaking, let $A\subset \mathbb{R}^d.$ Then $\Pi_x(A)$ is what an observer can see of $A$ at a certain position $x\in\mathbb{R}^d.$ In what follows, we say that a list of numbers $a_1,\dots,a_d$ are multiplicatively independent if they are not $0$ nor $1$ and $1,\log a_2/\log a_1,\dots,\log a_d/\log a_1$ are linearly independent over the field of rational numbers. See Section \ref{SelfSimilar} for the definition of self-similar sets and the open set condition. See \cite[Chapter 2]{Fa} for the definition of  Hausdorff dimension ($\Haus$).
	
	\begin{conj}\label{RadialConj}
		Let $A\subset\mathbb{R}^d,d\geq 2$ be a Cartesian product of self-similar sets in $\mathbb{R}$ with the open set condition and uniform contraction ratios. Suppose further that the contraction ratios are multiplicatively independent. If $\Haus A>d-1,$ then $\Pi_x(A)$ has non-empty interior for all $x\in\mathbb{R}^d.$
	\end{conj}
	In Section \ref{Projection}, we will provide more details. Conjecture \ref{RadialConj} turns out to be closely related to Graham's question. We do not need the full strength of this conjecture. A special case will be enough. See Conjecture \ref{conj: radial missing digits}.
	\begin{thm}\label{Number}
		Assuming Conjecture \ref{RadialConj} and Schanuel's conjecture, there are infinitely many integers $n$ such that $\binom{2n}{n}$ is coprime with $3\times 5\times 7.$ 
	\end{thm}
	\begin{rem}
		We do not need the full strength of Schanuel's Conjecture. It is enough only to assume that
		\[
		1,\frac{\log 3}{\log 5}, \frac{\log 3}{\log 7}
		\]
		are $\mathbb{Q}$-linearly independent.
	\end{rem}
	Currently, we are not able to prove Conjecture \ref{RadialConj}. Nonetheless, the strategy for proving Theorem \ref{Number} can be adapted to prove many other (unconditional) results concerning numbers with restricted digits. 
	
	First, we prove a quantitative version of a result in \cite{EGRS75}. 
	\begin{thm}\label{EGRSII}
		Let $p,q$ be two different odd primes. Consider the set $$A=\left\{n\in\mathbb{N}: \mathrm{gcd}\left(pq,\binom{2n}{n}\right)=1\right\}.$$ Then there is a constant $c>0$ depending on $p,q$ such that
		\[
		A\cap [1,N]\geq c\log N
		\]
		holds for all large enough integers $N$.
	\end{thm}
	Next, we prove a result concerning linear forms of  numbers with restricted digits.
	\begin{thm}\label{EGRSIII}
		There are infinitely many integers triples $(x,y,z)\in N_{3}^{\{0,1\}}\times N_{4}^{\{0,1\}}\times N_{5}^{\{0,1\}}$ with 
		\[
		x+y=z.
		\]
	\end{thm}
	\begin{rem}
		This result says that there are infinitely many sums of powers of five that can be written as sums of powers of three and four. We list a few examples:
		\[
		5=4+1,
		\]
		\[
		5^2=4^2+3^2,
		\]
		\[
		5^3+5^2=3^4+4^3+4+1,
		\]
		\[
		5^4+5^2=3^5+4^4+3^4+4^3+4+1+1.
		\]
		It is perhaps possible that all large enough integers can be written as a sum of form $N^{\{0,1\}}_3+N^{\{0,1\}}_4.$ If so, this theorem would follow as a direct consequence. We believe that the conclusion does not hold for general triples $(b_1,b_2,b_3)$ in the place of $(3,4,5).$ For example, we suspect that there are only finitely many integer triples $(x,y,z)\in N_{9}^{\{0,1\}}\times N_{10}^{\{0,1\}}\times N_{11}^{\{0,1\}}$ with $x+y=z.$ A partial result towards this direction is \cite[Theorem 1.5]{Y20} which says that for each $\epsilon>0$ and large enough integer $N,$ the amount of integer triples in $[1,N]^3$ with the above property is $O(N^{\epsilon}).$
	\end{rem}
	Another very natural question to consider is whether there are infinitely numbers with missing digits in many different bases at the same time. For example, are there integers $b_1>\dots>b_{100}>2$ such that there are infinitely many integers whose base $b_1,\dots,b_{100}$ expansions do not have digit zero? Our next result answers this question.
	\begin{thm}\label{EGRS4}
		Let $k\geq 2$ be an integer. Then there is an integer $M\geq 1$ such that for all $k$-tuples of multiplicative independent integers $b_1,\dots,b_k$ that are at least $M$, there are infinitely many integers whose base $b_1,\dots,b_k$ expansions all omit the digit zero.
	\end{thm}
	\begin{rem}
		The missing digit zero is not a special choice. In fact, one can choose an integer $m\geq 1$ and consider digit expansions in different bases with an arbitrary choice of $m$ missing digits for each base. Of course, this only makes sense if the bases in consideration are all greater than $m.$
		
		This brings us closer to Graham's question. However, we are still far away. Kummer's theorem indicates that for considering prime divisors of binomial coefficients, one needs to consider numbers that miss at least half of the digits in many different prime bases.
	\end{rem}
	To prove the above results, we use Newhouse's gap lemma. See Section \ref{Thickness}. It is possible to prove Theorems \ref{EGRSIII}, \ref{EGRS4} without having the geometry of self-similar sets in mind. In fact, it is possible to prove those results by directly using arithmetic of integers and some well-known facts of irrational rotations just like the original arguments in \cite[Lemma on page 84]{EGRS75}. For example, \cite[Lemma on page 84]{EGRS75} can be viewed as an integer version of Newhouse's gap lemma applied to 'integer self-similar sets'. The point of using the fractal geometric point of view in this paper is to make the ideas behind the proofs more transparent. In fact, the core of the proof of Theorem \ref{EGRSII} essentially uses the same ideas as in \cite[Lemma on page 84]{EGRS75} but expresses them in a more geometric way.
	
	\section{Preliminaries}\label{GMT}
	\subsection{A remark for Schanuel's conjecture}
	We need Schanuel's conjecture because in some of the proofs, we will use properties of (irrational) rotations on torus. Let $n\geq 1$ be an integer. Let $\mathbb{T}^n=\mathbb{R}^n/\mathbb{Z}^n$. Let $\alpha=(\alpha_1,\dots,\alpha_n)\in\mathbb{R}^n.$ We define the map $T_\alpha:\mathbb{T}^n\to\mathbb{T}^n$ to be
	\[
	x\in\mathbb{T}^n\to T_\alpha(x)=x+\alpha.
	\]
	We say that $T_\alpha$ is the rotation on $\mathbb{T}^n$ with the rotation angle $\alpha.$ If $1,\alpha_1,\dots,\alpha_n$ are $\mathbb{Q}$-linearly independent, then we say that $T_\alpha$ is an irrational rotation. In this case, it is well-known that for each $x\in\mathbb{T}^n,$
	\[
	\overline{\{T^k_\alpha(x)\}_{k\geq 1}}=\mathbb{T}^n.
	\]
	See \cite[Corollary 4.15]{EW}.
	
	Let $b_1,b_2\dots,b_n$ be $n\geq 2$ integers. In order to study digit expansions with respect to these integers, it is often useful to consider the rotation on $\mathbb{T}^{n-1}$ with the rotation angle,
	\[
	\Lambda=(\log b_1/\log b_2,\dots,\log b_1/\log b_n).
	\]
	To do this, it is useful to know whether the above vector generates an irrational rotation. This is the case if
	\[
	\Lambda'=\left(\frac{\prod_{i=1}^{n} \log b_i}{\log b_1},\dots, \frac{\prod_{i=1}^{n} \log b_i}{\log b_n}\right)
	\]
	are $\mathbb{Q}$-linearly independent. If $n=2$, then the situation is simple. For $n\geq 3,$ the problem becomes challenging.  For example, it is not known whether $1,\log 2/\log 3$, $\log 2/\log 5$ are $\mathbb{Q}$-linearly independent. Problems of this kind are related to Schanuel's conjecture, see \cite{A71}.
	\begin{conj}[Schanuel]\label{schanuel}
		Let $n\geq 2$ be an integer. Let $x_1,\dots,x_n$ be  $\mathbb{Q}$-linearly independent complex numbers. Then the transcendence degree of $$\mathbb{Q}(x_1,\dots,x_n,e^{x_1},\dots,e^{x_n})$$ is at least $n$. 
	\end{conj}
	In case when $e^{x_1},\dots,e^{x_n}$ are integers, the conjectures reduces to saying that $x_1,\dots,x_n$ are algebraically independent over $\mathbb{Q}.$ This implies that $\Lambda'$ is indeed $\mathbb{Q}$-linearly independent if $1,\log b_1,\dots,\log b_n$ are $\mathbb{Q}$-linearly independent. This is the reason that whenever we are dealing with digit expansions with more than two bases, Schanuel's conjecture is likely to appear. We are not using the full power of Schanuel's conjecture here. However, this conjecture is so difficult that even the homogeneous quadratic case is not known. The only known result in this direction is Baker's theory on linear forms of logarithms, see also \cite{A71} and the references therein for more details.
	
	We should nonetheless remark that although we are dealing with digit expansions with more than two bases in this paper, the $\mathbb{Q}$-linear independence of $\Lambda'$ is not always involved. In fact, we will only need Schanuel's conjecture for proving Theorem \ref{Number}. For Theorem \ref{EGRSIII}, the proof would be much simpler by assuming Schanuel's conjecture. Additional efforts need to be taken to get rid of it. For Theorem \ref{EGRS4}, we simply do not meet the situation where Schanuel's conjecture is needed.

	\subsection{Self-similar sets  and the open set condition}\label{SelfSimilar}
	Let $\mathcal{F}=\{f_i\}_{i\in\Lambda}$ be a finite collection of linear maps on $\mathbb{R}.$ We can write down each linear maps explicitly as $f_i(x)=r_ix+a_i.$ We assume that $r_i\in (0,1)$ for all $i\in\Lambda.$ We call such a collection of linear maps to be a \emph{linear IFS}. The parameters $r_i,i\in\Lambda$ are called \emph{contraction ratios} and $a_i,i\in\Lambda$ are called \emph{translations}. In case when all the contraction ratios are equal to $r\in (0,1)$, we call $r$ to be the \emph{uniform contraction ratio}. 
	
	By \cite{H81}, there is a unique non-empty compact set $F$ such that
	\[
	F=\bigcup_{i\in\Lambda} f_i(F).
	\]
	We call such a set $F$ to be a \emph{self-similar set} determined by $\mathcal{F}$. We say that $\mathcal{F}$ satisfies the \emph{open set condition} if there is a bounded open set $U\subset\mathbb{R}$ such that $f_i(U)\subset U$ for each $i\in\Lambda$ and
	\[
	f_i(U)\cap f_j(U)=\emptyset
	\]
	as long as $i\neq j.$ The open set condition is a condition on $\mathcal{F}$ rather than $F.$ However, we also say that $F$ satisfies the open set condition if there is an IFS $\mathcal{F}$ such that $\mathcal{F}$ determines $F$ and has the open set condition.
	
    Let $b>1$ be an integer and $B\subset \{0,\dots,b-1\}.$ Let $A_b^B$ be the set
	\begin{align*}
	\overline{\{x\in [0,1]\setminus\mathbb{Q}: \text{ the $b$-ary expansion of $x$ contains only digits in $B$} \}}\\
	={\{x\in [0,1]: \text{ some $b$-ary expansion of $x$ contains only digits in $B$} \}}.
	\end{align*}
	Then $A_b^B$ is self-similar with $\Lambda=B$, and for $z\in B,$ the linear map $f_z$ is as follows,
	\[
	x\in\mathbb{R}\to f_z(x)=\frac{1}{b}x+\frac{z}{b}.
	\]
	To verify that $\mathcal{F}=\{f_z\}_{z\in B}$ has the open set condition, we can choose $U=(0,1).$ A famous example of this kind is the middle third Cantor set $A^{\{0,2\}}_3$.
	\subsection{Thickness, intersection and sums of Cantor sets}\label{Thickness}
	Let $A\subset\mathbb{R}$ be a compact, totally disconnected set without isolated points. We shall call $A$ a Cantor set. First, we assume that $A\subset [0,1]$ and the convex hull of $A$ is $[0,1].$ In this case, we see that $\mathbb{R}\setminus A$ is a countable union of disjoint open intervals $\{I_i\}_{i\geq 1}$. Notice that $(-\infty,0), (1,\infty)$ are the only unbounded intervals  in $\{I_i\}_{i\geq 1}.$ We call bounded intervals in $\{I_i\}_{i\geq 1}$ to be \emph{bounded gaps} of $A.$ Thus, $A$ can be constructed by iteratively chopping out open intervals from $\mathbb{R}.$ Let $I=(a,b)$ be one of those bounded open intervals. We find the interval $I^{-}\in\{I_i\}_{i\geq 0}$ and $I^-\subset (-\infty, a)$ such that $|I^{-}|\geq |I|$ and there is no other such intervals between $I^{-}$ and $I.$ Similarly, we can find $I^+\subset (b,\infty)$ to the right of $I.$ Suppose that $I^-=(c,d)$ and $I^+=(e,f).$ We see that
	\[
	-\infty\leq c<d<a< b<e<f\leq \infty.
	\]
	Let $b_L=a-d$, $b_R=e-b$ and
	\[
	C(I)=\min\{b_L,b_R\}/|I|.
	\]
	We define $C(A)=\inf_{I\in\{I_i\}_{i\geq 1}, I \text{ bounded}} C(I).$ This number $C(A)$ is called \emph{the thickness} of $A.$ We also define \emph{the normalized thickness} of $A$ to be
	\[
	S(A)=\frac{C(A)}{C(A)+1}.
	\]
	For a general Cantor set $A$, we can perform a uniquely determined orientation preserving affine transformation $T$ which maps the convex hull of $A$ to the unit interval. Then we define 
	\[
	C(A)=C(T(A)), S(A)=S(T(A)).
	\]
	We have the following result due to Newhouse. See \cite{N79}.
	
	\begin{thm}\label{NEWHOUSE}[Newhouse's gap lemma]
		Let $A,B$ be two compact, totally disconnected sets. Suppose that $A$ is not contained in any of the gaps of $B$ and vice versa. If $S(A)+S(B)\geq 1,$ then $A\cap B\neq\emptyset.$
	\end{thm}
	Newhouse proved the above result with the condition $S(A)+S(B)>1.$ Astels \cite[Theorem 2.2]{A99} showed that the above result holds for $S(A)+S(B)=1$ as well.
	
	In the case when we have two very thick Cantor sets, we expect that their intersection could also be very thick.  In this direction, we have \cite[Theorem 1, and the discussion at the beginning of page 882, and the discussion before Corollary 6]{HKY}. 
	\begin{thm}\label{NEWHOUSEMULTI}
		Let $A,B$ be two compact, totally disconnected sets. Suppose that $A$ is not contained in any of the gaps of $B$ and vice versa. For each $\delta>0,$ there is an $\epsilon>0$ such that if $S(A), S(B)$ are greater than $1-\epsilon,$ then there is a compact set $C\subset A\cap B$ such that $S(C)$ is greater than $1-\delta.$ Moreover, suppose that the convex hulls $Conv(A),Conv(B)$ are such that $Conv(A)\cap Conv(B)$ contains neither $A$ nor $B$. Then by making $\epsilon$ smaller if necessary, $C$ can be chosen so that $|Conv(C)|$ is at least $(1-\delta)|Conv(A)\cap Conv(B)|.$
	\end{thm}
	Recently, there are some further development on the intersecting structures of thick Cantor sets. See \cite{Ya}.
	
	The following result was proved in \cite[Theorem 2.4]{A99}.
	
	\begin{thm}\label{NEWHOUSE2}
		Let $A_1,\dots,A_k$ be $k\geq 2$ Cantor sets. Suppose that their convex hulls are $I_1,I_2,\dots,I_k$ and the size of their largest bounded gaps are $g_1,\dots,g_k$ respectively. Suppose further that $\sum_{i=1}^k S(A_i)\geq 1$ and $\min\{|I_1|,\dots,|I_k|\}> \max\{g_1,\dots,g_k\}.$ Then $A_1+\dots+A_k$ is an interval.
	\end{thm}
	
	For later use, we reformulate the above result in terms of intersections. Let $A_1,A_2,A_3$ be three Cantor sets satisfying the hypothesis of the above theorem. Consider the Cartesian product $A=A_1\times A_2\times A_3.$ Let $v\in S^{2}$ be a direction vector of $\mathbb{R}^3$ and let $H_v$ be the plane passing through the origin and normal to $v.$ The family of planes parallel to $H_v$ can be parametrized by $\mathbb{R},$ more precisely, $\{H_v(a)=H_v+av\}_{a\in\mathbb{R}}.$ Denote $P_v:\mathbb{R}^3\to \mathbb{R}v$ to be the corresponding orthogonal projection in direction $v$. The above theorem says that $P_{(1/\sqrt{3},1/\sqrt{3},1/\sqrt{3})}(A)$ is an interval. Moreover, notice that the normalized thickness is invariant under affine maps. In addition, there is a non-empty open neighbourhood $O$ of $(1/\sqrt{3},1/\sqrt{3},1/\sqrt{3})$ in $S^2$ such that whenever $v=(v_1,v_2,v_3)\in O,$ the size of the minimum convex hull of $v_1A_1, v_2A_2, v_3A_3$ is larger than the size of their maximum gap. This implies that $P_v(A)$ is an interval. 
	
	Now if $\Pi_v(A)$ is an interval, then we see that 
	\[
	\{a\in\mathbb{R}: H_v(a)\cap A\neq\emptyset\}
	\]
	is an interval. More precisely, if $H_v(a)\cap I_1\times I_2\times I_3\neq\emptyset$ then we have $H_v(a)\cap A\neq\emptyset.$ From here, we have the following corollary.
	
	\begin{cor}\label{Slice}
		Let $A_1,A_2,A_3$ be $3$ Cantor sets with $S(A_1)+S(A_2)+S(A_3)\geq 1.$ Suppose that their convex hulls are $I_1,I_2,I_3$ and the size of their largest gaps are $g_1,g_2,g_3$ respectively. If \[\min\{|I_1|,|I_2|,|I_3|\}>\max\{g_1,g_2,g_3\}\] then there is a non-empty open set $O\subset S^2$ such that whenever $v\in O,$ we have
		\[
		H_v(a)\cap I_1\times I_2\times I_3\neq\emptyset\implies H_v(a)\cap A_1\times A_2\times A_3\neq\emptyset.
		\]
	\end{cor}
	
	We now compute the normalized thickness of some examples of Cantor sets. First, let $b>2$ be an integer and let $B=\{0,1,\dots,l\}$ where $l<b-1.$ We consider the set
	\[
	A^{B}_b=\{x\in [0,1]: \text{ some $b$-ary expansion of $x$ contains only digits in $B$} \}.
	\]
	\begin{lma}\label{THICK}
		Let $b,B, A^B_b$ be as above. The normalised thickness of $A^B_b$ is 
		\[
		S(A^B_b)=\frac{l}{b-1}.
		\]
	\end{lma}
	\begin{proof}
		The convex hull of $A^B_b$ is $[0,a]$ where
		\[
		a=\sum_{i=1}^{\infty} \frac{l}{b^i}=\frac{l}{b-1}.
		\]
		The largest gaps of $A^B_b$ are of size 
		\[
		\frac{b-1-l}{b(b-1)}.
		\]
		Those gaps are located in each of the  intervals
		\[
		[0,1/b],\dots,[(l-2)/b,(l-1)/b].
		\]
		Let $I$ be one of those gaps. Then we see that
		\[
		C(I)=\frac{l}{b-1-l}.
		\]
		From here and the fact that $A^B_b$ is self-similar, we see that $C(A^B_b)=l/(b-1-l)$ and $S(A^B_b)=l/(b-1).$ 
	\end{proof}
	
	Next, we consider the middle third Cantor set $C_3=A^{\{0,2\}}_{3}.$ Following the above steps we see that $C(C_3)=1, S(C_3)=1/2.$ Let $k>1$. We consider the image of $C_3$ under the map $x\to x^k.$ We write this image as $C^k_3.$  We record here a simple observation which illustrates an application of thickness of Cantor set.
	\begin{thm}\label{Waring}\footnote{We were told by S. Chow that this result (for $k=2$) was conjectured in \cite[Conjecture 13]{ABT19} and answered in \cite{JLWZ19}. We thank him for providing the references.}
		For all $x\in [0,4]$, there exist $x_1,x_2,x_3,x_4\in C_3$ such that
		\[
		x=x_1^2+x_2^2+x_3^2+x_4^2.
		\]
		More generally, for each number $k>1$ there is an integer $1<n(k)\leq 2^k$ such that all $x\in [0,n(k)]$ can be written as
		\[
		x=\sum_{i=1}^{n(k)}x_i^k
		\]
		where $x_1,\dots,x_{n(k)}\in C_3.$
	\end{thm} 
	
	\begin{lma}\label{kpower}
		Let $k,C^k_3$ be as above. Then 
		\[
		S(C^k_3)=\frac{1}{2^{k}}.
		\]
	\end{lma}
	\begin{proof}
		Suppose that $I=(a,a+\Delta)$ is a bounded gap of $C_3.$ Then we see that the next gap on the right of $I$ which is not smaller than $I$ has left endpoint $a+2\Delta.$ Similarly, the next gap on the left of $I$ which is not smaller than $I$ has right endpoint $a-\Delta.$ For the middle third Cantor set $C_3,$ we always have $a\geq \Delta.$
		
		After taking the $k$-th power map, we have points
		\[
		(a-\Delta)^k, a^k, (a+\Delta)^{k}, (a+2\Delta)^k.
		\]  
		The gap $I$ is now transformed into a gap of size
		\[
		|a^k- (a+\Delta)^{k}|.
		\]
		This length is increasing as a function of $a$ as well as $\Delta.$ Thus we see that the number $b_L$ in the definition of thickness is at least $|(a-\Delta)^k-a^k|.$ We need to take care of $b_R.$ By the above argument we see that $b_R$ is at most $(a+2\Delta)^k-(a+\Delta)^k.$ However, we need a lower bound for $b_R.$ To do this, we need to study how the  gaps of $C_3$ in $[a+\Delta,a+2\Delta]$ are transformed under the map $x\to x^k.$ Notice that a gap inside $[a+\Delta, a+2\Delta]$ might become larger than $|a^k- (a+\Delta)^{k}|$ after taking the $k$-th power. Note that by the convexity of $x\to x^k,$ there is a $\delta\in (0,\Delta]$ such that
		\[
		(a+\Delta+2\delta)^k-(a+\Delta+\delta)^k=(a+\Delta)^k-a^k.
		\]
		Let $(a+\Delta+\alpha,a+\Delta+\alpha+\beta)$ be a bounded gap of $C_3\cap [a+\Delta,a+2\Delta]$ with $\alpha, \beta>0$ and $\alpha+\beta\leq\Delta.$ Then we have $\beta\leq \alpha.$ This follows by self-similarity, since $[a+\Delta,a+2\Delta]$ is a similar copy of $C_3$ with contraction factor $\Delta.$ From here we see that in order to find gaps of $C^k_3$ on the right side of $(a+\Delta)^k$ with length at least $|a^k-(a+\Delta)^k|,$ one has to search the gaps of   $C^k_3$ in $[(a+\Delta+\delta)^k,\infty).$  Thus $b_R$ is at least \[(a+\Delta+\delta)^k-(a+\Delta)^k.\]
		Observe that for $x=\delta/(a+\Delta),$
		\[
		\frac{b_R}{(a+\Delta)^k-a^k}\geq\frac{(a+\Delta+\delta)^k-(a+\Delta)^k}{(a+\Delta+2\delta)^{k}-(a+\Delta+\delta)^k}=\frac{(1+x)^k-1}{(1+2x)^k-(1+x)^k}.
		\]
		Since $0<\delta\leq\Delta\leq a,$ we see that the above is at least (the value when $x=1/2$)
		\[
		c_k=\frac{1.5^k-1}{2^k-1.5^k}.
		\]
		It follows that $b_R\geq c_k((a+\Delta)^k-a^k).$ We see that
		\begin{align}\label{**}
		\min\{b_L,b_R\}/|a^k- (a+\Delta)^{k}|\geq\min\left\{\frac{a^k-(a-\Delta)^k}{(a+\Delta)^{k}-a^k},c_k\right\}\geq \frac{1}{2^k-1}.
		\end{align}
		For the last inequality, notice that
		\[
		\frac{a^k-(a-\Delta)^k}{(a+\Delta)^{k}-a^k}=\frac{1-(1-\Delta/a)^k}{(1+\Delta/a)^k-1}.
		\]
		As $(a,a+\Delta)$ is a bounded gap of $C_3,$ we see that $0<\Delta\leq a.$ The function
		\[
		x\in [0,1]\to \frac{1-(1-x)^k}{(1+x)^k-1}
		\]
		takes the minimum at $x=1$ with the value $1/(2^k-1).$ For all $k>1,$ we have
		\[
		\frac{1}{2^k-1}\leq \frac{1.5^k-1}{2^k-1.5^k}.
		\]
		From here we conclude (\ref{**}). As (\ref{**}) holds for all bounded gaps of $C^k_3$ we see that
		\[
		C(C^k_3)\geq \frac{1}{2^k-1}
		\]
		and
		\[
		S(C^k_3)\geq \frac{1}{2^k}.
		\]
		On the other hand, let $n\geq 1$ be an integer. Then $I=(3^{-n},2\times 3^{-n})$ is a bounded gap of $C_3.$ Now, in $C_3,$ the next gap on the left of $I$ with length at least $|I|$ is an infinite gap, i.e. $(-\infty,0).$ Thus in $C_3^k,$ the next gap on the left of $(3^{-kn},2^{k}3^{-kn})$ with at least the same length is again $(-\infty,0).$ This shows that the inequality (\ref{**}) is sharp and the proof concludes.
	\end{proof}
	From here, we see that Theorem \ref{Waring} follows.
	
	\begin{proof}[Proof of Theorem \ref{Waring}]
		By Lemma \ref{kpower} and Theorem \ref{NEWHOUSE2}, it is enough to check the gap conditions stated in Theorem \ref{NEWHOUSE2}. As the convex hull of $C_3^k$ is $[0,1]$ and the largest gap is strictly shorter than $1$, the result follows.
	\end{proof}
	
	\section{Radial projections of fractal sets: general overview}\label{Projection}
	Let $d\geq 2$ be an integer. Let $x\in\mathbb{R}^d$ be a point. Recall that $\Pi_x$ is defined as follows
	\[
	\Pi_x(y)=\frac{y-x}{|x-y|}\in S^{d-1}
	\] 
	for $y\neq x.$  The following result was proved in \cite{O19}.
	\begin{thm}\label{Radial}
		Let $d\geq 2$ be an integer. Let $A\subset\mathbb{R}^d$ be a Borel set with $\Haus A>d-1.$ Then $\Pi_x(A)$ has positive Lebesgue measure for almost all $x\in\mathbb{R}^d.$
	\end{thm}
	Intuitively speaking, the above theorem says that if $A\subset\mathbb{R}^d$ is large enough, then $\Pi_x(A)$ should also be large enough, at least generically. If the set $A$ is a Cartesian product of self-similar sets with the open set condition, we believe that a much stronger result should hold. See Conjecture \ref{RadialConj}. In our situation, it is convenient to explicitly state a special case.
	\begin{conj}\label{conj: radial missing digits}
		Let $d>1$ be an integer. Let $b_1,\dots,b_d>1$ be integers. For each $i\in\{1,\dots,d\}$,  let $B_1\in \{0,\dots,b_i-1\}$ be a choice of digits in base $b_i$ and
		\[
		A_i=\{x\in [0,1]: \text{ some $b_i$-ary expansion of $x$ contains only digits in $B_i$} \}.
		\] 
		Consider the set $A=A_1\times\dots\times A_d.$ If $\Haus A>d-1,$ then $\Pi_x (A)$ has positive Lebesgue measure for all $x\in\mathbb{R}^d.$ If moreover $b_1,\dots,b_d$ are multiplicatively independent, then $\Pi_x(A)$ contains non-empty interior for all $x\in\mathbb{R}^d.$
	\end{conj}
	As $A_1,\dots,A_d$ are self-similar sets with the open set condition, this conjecture is indeed a special case of Conjecture \ref{RadialConj} under the multiplicative independence of $b_1,\dots,b_d.$ We believe that for the positivity of the Lebesgue measure, it is not necessary to require the multiplicative independence of $b_1,\dots,b_d.$ Currently, almost nothing is known towards Conjectures \ref{RadialConj}, \ref{conj: radial missing digits}. For Conjecture \ref{RadialConj}, see\cite{HS12} or \cite{Sh} for related results. Although results in \cite{HS12}, \cite{Sh} are for linear projections rather than radial projections, it is possible to use the arguments to prove results for linear projections. For example, let $A$ be as in Conjecture \ref{RadialConj} with $d=2.$ Results in \cite{Sh} provide estimates of the size $l^\delta\cap A$ ($l^\delta$ is the $\delta$-neighbourhood of $l$) uniformly across $\delta>0$ and lines $l\subset\mathbb{R}^2$ with direction strictly away from being parallel with the coordinate axes. In particular, for each $\epsilon>0,$ for all small enough $\delta>0$ and all lines $l$, to cover $l^\delta\cap A,$ it is enough to use at most $\delta^{-(\Haus A-1+\epsilon)}$ many $\delta$-balls. Using this estimate, we can conclude that $\Pi_{(0,0)}A$ cannot be small. In fact, suppose that $\Pi_{(0,0)}A$ can be covered by at most $0.0001\delta^{-(1-2\epsilon)}$ many $\delta$-balls. Then we are forced to have a fibre of $\Pi_{(0,0)}(A)$ which cannot be covered by
	\[
	\delta^{-\Haus A}/\delta^{-(1-2\epsilon)}=\delta^{-(\Haus A-1+2\epsilon)}
	\]
	many $\delta$-balls. Fibres of $\Pi_{(0,0)}(A)$ are lines passing through the origin. In particular, this contradicts with the results in \cite{Sh}. From here, we conclude that $\boxd \Pi_{(0,0)}A=1.$ Here $\boxd \Pi_{(0,0)}A$ is the box dimension of $\Pi_{(0,0)}A$. It is possible to upgrade this result to $\Haus \Pi_{(0,0)}A=1$ by carefully studying a certain regular measure $\mu$ supported on $A$ and its image under $\Pi_{(0,0)}.$ Of course, $(0,0)$ can be replaced with any other point in $\mathbb{R}^2.$
	
	\section{Proofs of Theorems \ref{Number}, \ref{EGRSII}, \ref{EGRSIII} and \ref{EGRS4}}
	Now we prove Theorem \ref{Number}. 
	
	\begin{proof}[Proof of Theorem \ref{Number}]
		Consider the set $N_3^{B_3}\times N_5^{B_5}\times N_7^{B_7}$ where
		\[
		B_3=\{0,1\}, B_5=\{0,1,2\}, B_7=\{0,1,2,3\}.
		\]
		Now we construct the sets
		\[
		A_3=\{x\in [1,3]: \text{ some 3-ary expansion of $x$ contains only digits in $B_3$}  \},
		\]
		\[
		A_5=\{x\in [1,5]: \text{ some 5-ary expansion of $x$ contains only digits in $B_5$}  \},
		\]
		\[
		A_3=\{x\in [1,7]: \text{ some 7-ary expansion of $x$ contains only digits in $B_7$}  \}.
		\]
		Then we see that (\cite[Section 7.1]{Fa}) $\Haus A_3\times A_5\times A_7=\log 2/\log 3+\log 3/\log 5+\log 4/\log 7>2.$ Let $\{\}$ be the fractional part symbol, i.e. for  a real number $a,$ $\{a\}\in [0,1)$ is the unique number $t$ in $[0,1)$ so that $t-a\in\mathbb{Z}.$ For each integer $k\geq 1$, consider the line $l_k$ passing through the origin with direction vector
		\[
		(1,5^{\{k\log 3/\log 5\}},7^{\{k\log 3/\log 7\}}).
		\]
		If $l_k\cap A_3\times A_5\times A_7\neq\emptyset$ then we take a point $(x,y,z)\in l_k\cap A_3\times A_5\times A_7.$ Consider the point
		\[
		(x',y',z')=(3^k x, 5^{[k\log 3/\log 5]} y, 7^{[k\log 3/\log 7]} z).
		\]
		Since $y=5^{\{k\log 3/\log 5\}} x, z=7^{\{k\log 3/\log7\}}x$, so we see that
		\[
		5^{[k\log 3/\log 5]} y=5^{k\log 3/\log 5}x=3^k x, 7^{[k\log 3/\log 7]} z=7^{k\log 3/\log 7}x=3^k x.
		\]
		Thus we see that $x'=y'=z'.$ It is straightforward to see that the $3$-ary expansion of $x'$ contains only digits in $B_3$, the $5$-ary expansion of $y'$ contains only digits in $B_5$ and the $7$-ary expansion of $z'$ contains only digits in $B_7.$ Taking the integer part, we see that
		\[
		[x']=[y']=[z'] \in N_{3,5,7}^{B_3,B_5,B_7}.
		\]
		
		Under Conjecture \ref{RadialConj}, we see that there are infinitely many integers $k\geq 1$ such that
		\[
		l_k\cap A_3\times A_5\times A_7\neq\emptyset.
		\]
		Indeed, the direction vector of $l_k$ is
		\[
		(1,5^{\{k\log 3/\log 5\}},7^{\{k\log 3/\log 7\}}).
		\]
		Under Schanuel's conjecture,  $1,\log 3/\log 5, \log 3/\log 7$ are linearly independent over $\mathbb{Q},$ and therefore the sequence
		\[
		(\{k\log 3/\log 5\},\{k\log 7/\log 5\})_{k\geq 0}
		\]
		equidistributes in $[0,1]^2.$ Thus the closure of $\cup_{k\geq 0} l_k$ contains the cone $C$ spanned by the origin and the set
		\[
		\{1\}\times [1,5]\times [1,7].
		\]
		Now if $\Pi_{(0,0,0)}(A_3\times A_5\times A_7 \cap C )$ has non-empty interior, then $l_k$ will intersect $A_3\times A_5\times A_7 \cap C$ infinitely often. Observe that $A_3\times A_5\times A_7$ intersects the interior of $C.$ Thus there is a point $a\in A_3\times A_5\times A_7$ and a $r>0$ such that $B_a(r)\subset C.$ Since $A_3\cap [1,2],A_5\cap [1,2],A_7\cap [1,2]$ are self-similar with contraction ratios $1/3,1/5,1/7$ respectively, it is possible to find linear maps $f_3,f_5,f_7:\mathbb{R}\to\mathbb{R}:$
		\[
		f_3(x)=3^{-l_3}x+t_3,f_5(x)=5^{-l_5}x+t_5,f_7(x)=7^{-l_7}x+t_7
		\]
		for some positive integers $l_3,l_5,l_7$ and real numbers $t_3,t_5,t_7$ such that
		\[
		f_3(A_3\cap [1,2])\subset A_3, f_5(A_5\cap [1,2])\subset A_5, f_7(A_7\cap [1,2])\subset A_7,
		\]
		and
		\[
		f_3\times f_5\times f_7 (A_3\times A_5\times A_7\cap [1,2]^3)\subset B_a(r).
		\]
		We denote $A'_3=f_3(A_3\cap [1,2]), A'_5=f_5(A_5\cap [1,2]), A'_7=f_7(A_7\cap [1,2]).$ Then we see that
		\[
		A'_3\times A'_5\times A'_7\subset A_3\times A_5\times A_7\cap C.
		\]
		Moreover, since $f_3\times f_5\times f_7$ is linear and invertible and thus bi-Lipschitz,  by \cite[Corollary 2.4]{Fa}, we have
		\[
		\Haus (A'_3\times A'_5\times A'_7)=\Haus(A_3\times A_5\times A_7\cap [1,2]^3).
		\]
		Next, observe that $A_3\times A_5\times A_7\cap [1,2]^3$ is a translation of $K_3\times K_5\times K_7$ where
		\[
		K_3=\{x\in [0,1]: \text{ some 3-ary expansion of $x$ contains only digits in $B_3$}   \},
		\]
		\[
		K_5=\{x\in [0,1]: \text{ some 5-ary expansion of $x$ contains only digits in $B_5$}   \},
		\]
		\[
		K_7=\{x\in [0,1]: \text{ some 7-ary expansion of $x$ contains only digits in $B_7$}   \}.
		\]
		By \cite[Section 7.1]{Fa}, we see that $\Haus K_3\times K_5\times K_7=\Haus A_3\times A_5\times A_7.$
		Now Conjecture \ref{RadialConj} tells us that $\Pi_{(0,0,0)}(A'_3\times A'_5\times A'_7)$ has non-empty interior. This proves the result.
	\end{proof}
	Theorem \ref{EGRSII} follows by using a similar argument.
	\begin{proof}[Proof of Theorem \ref{EGRSII}]
		The proof is very similar to the previous one. Let $p,q$ be two distinct odd primes. Then we see that $\log p/\log q$ is irrational. Consider the set $N^{B_p}_p\times N^{B_q}_q$ where
		\[
		B_p=\{0,1,\dots,(p-1)/2\}, B_q=\{0,1,\dots,(q-1)/2  \}.
		\]
		We also construct the sets
		\[
		A_p=\{x\in [1,p]: \text{ some $p$-ary expansion of $x$ contains only digits in $B_p$}   \},
		\]
		\[
			A_q=\{x\in [1,q]: \text{ some $q$-ary expansion of $x$ contains only digits in $B_q$}   \}.
		\]
		
		Now by Lemma \ref{THICK} (applied to suitable affine copies of $A_p,A_q$) we see that $S(A_p)=S(A_q)=1/2.$ We now use Theorem \ref{NEWHOUSE2}. Since $S(A_p)+S(A_q)=1,$ we see that the difference set $A_p-A_q$ is an interval (for the gap condition, observe that the largest gaps of $A_p,A_q$ are shorter than the unit interval). This says that $P_{(-1/\sqrt{2},1/\sqrt{2})}(A_p\cap A_q)$ is an interval. The convex hull of $A_p\cap A_q$ is a rectangle not contained in only one side of the line $\{x=y\}.$ This implies that $P_{(-1/\sqrt{2},1/\sqrt{2})}(A_p\cap A_q)$ contains $(0,0)$ as an interior point. Since the gap condition is preserved under linear perturbations close to the identity, as a result, we see that there is an non-trivial open set $O\subset S^1$ containing $(-1/\sqrt{2},1/\sqrt{2})$ such that if $v\in O,$ the linear projection $P_v(A_p\times A_q)\subset\mathbb{R}v$ is an interval containing $(0,0).$ This implies that $l_{v^{\perp}}\cap A_p\cap A_q\neq\emptyset.$ Thus we see that $\Pi_{(0,0)}(A_p\times A_q)$ contains non-empty interior around $(1/\sqrt{2}),1/\sqrt{2})$. We can now apply the same argument as in the proof of Theorem \ref{Number}. For each integer $k\geq 1,$ consider the line $l_k$ passing through the origin with direction vector $(1,q^{\{k\log p/\log q\}}).$ If $l_k\cap A_p\times A_q\neq\emptyset$ then we find a point $(x,y)\in l_k\cap A_p\times A_q.$ Next, consider the point
		\[
		(x',y')=(p^kx,q^{[k\log p/\log q]}y).
		\]
		As before, we see that
		\[
		[x']=[y']\in N^{B_p,B_q}_{p,q}.
		\]
		After replacing $A_p\times A_q$ with a suitable affine copy $A'_p\times A'_q$ if necessary as in the previous proof, we see that there is an interval $I\subset [0,1]$ such that whenever $\{k \log p/\log q\}\in I$, there is a number $n\in N^{B_p,B_q}_{p,q}\cap [p^{k},p^{k+1}].$  Since  $\{k \log p/\log q\}\in I$ happens for $k$ inside a subset of integers with positive density, we see that there is a $c>0$ and for all large enough integers $N,$ there are least $cN$ many intervals among
		\[
		[1,p), [p,p^{2}),\dots,[p^{N-1},p^N)
		\]
		intersecting $N^{B_p,B_q}_{p,q}.$ Thus the result follows by applying Kummer's theorem (Theorem \ref{thm: Kummer}).
	\end{proof}
	At this stage, Theorem \ref{EGRSIII} seems to be clear, at least under Schanuel's conjecture. We will first prove this theorem under Schanuel's conjecture and then explain how to get rid of it.
	\begin{proof}[Proof of Theorem \ref{EGRSIII}]
		Consider the set $N_3^{B_3}\times N_4^{B_4}\times N_5^{B_5}$ where
		\[
		B_3=\{0,1\}, B_4=\{0,1\}, B_5=\{0,1\}.
		\]
		Now we construct the self-similar sets
		\[
		A_3=\{x\in [1/3,2/3]: \text{ some $3$-ary expansion of $x$ contains only digits in $B_3$}   \},
		\]
		\[
		A_4=\{x\in [1/4,1/2]: \text{ some $4$-ary expansion of $x$ contains only digits in $B_4$}   \},
		\]
		\[
		A_5=\{x\in [1/5,2/5]: \text{ some $5$-ary expansion of $x$ contains only digits in $B_5$}   \}.
		\]
		By Lemma \ref{THICK} (applied to suitable affine copies of $A_3,A_4,A_5$), we see that $S(A_3)=1/2, S(A_4)=1/3$ and $S(A_5)=1/4.$ Thus we see that
		\begin{align}\label{eqn: @}
		S(A_3)+S(A_4)+S(A_5)=\frac{13}{12}>1.
		\end{align}
		The convex hulls of $A_3,A_4,A_5$ are $[1/3,1/2],[1/4,1/3],[1/5,1/4].$ Let $j_1,j_2$ be integers. For the time being, we treat $j_1, j_2$ as being general. Later on, we will see that for the proof of this theorem, it is enough to choose $j_1=j_2=1.$ Let $k\geq 1$ be an integer and let $H_{j_1,j_2,k}$ be the plane
		\[
		\{x+4^{-\{k\log 3/\log 4\}+j_1}y-5^{-\{k\log 3/\log 5\}+j_2}z=0\}.
		\]
		Suppose that $H_{j_1,j_2,k}\cap A_3\times A_4\times A_5\neq\emptyset.$ We take a point $(x,y,z)$ in this intersection. Consider the point
		\[
		(x',y',z')=(3^k x, 4^{[k\log 3/\log 4]+j_1} y, 5^{[k\log 3/\log 5]+j_2} z).
		\]
		Since we have
		\[
		x+4^{-\{k\log 3/\log 4\}+j_1}y-5^{-\{k\log 3/\log 5\}+j_2}z=0
		\]
		we see that
		\[
		3^{-k}x'+4^{-\{k\log 3/\log 4\}-[k\log 3/\log 4]}y'-5^{-\{k\log 3/\log 5\}-[k\log 3/\log 5]}z'=0
		\]
		Thus we have
		\[
		x'+y'-z'=0.
		\]
		We assume Schanuel's conjecture for the time being. Later on, we will remove this dependence. 
		
		From (\ref{eqn: @}) and the discussion above Corollary \ref{Slice}, we claim that there exist $j_1,j_2$ so that $H_{j_1,j_2,k}$ intersects $A_3\times A_4\times A_5$ for infinitely many integers $k.$ More precisely, consider the plane
		\[
		H=\{x+k_1y-k_2z=0\}
		\]
		where $k_1\in [4^{j_1-1},4^{j_1}], k_2\in [5^{j_2-1},5^{j_2}]$. It's normal vector is $(1,k_1,-k_2).$ Then, as long as
		\begin{align}\label{eqn: Q}
		1/2+k_1/3-k_2/5\geq 0\geq 1/3+k_1/4-k_2/4,
		\end{align}
		the Cartesian product of the convex hulls of $A_3,A_4,A_5$ intersects the plane $H.$ In this case, we see that $H\cap A_3\times A_4\times A_5$ is not empty if the minimal hull/maximal gap condition in Corollary \ref{Slice} is satisfied. The lengths of the convex hulls of $A_3,A_4,A_5$ are $1/6,1/12,1/20.$ The lengths of the largest gaps are $1/18,1/24,3/100.$ The minimal hull/maximal gap condition can be now written as
		\begin{align}\label{eqn: QQ}
		\min\{1/6,k_1/12,k_2/20\}\geq \max\{1/18,k_1/24,3k_2/100\}.
		\end{align}
		This is equivalent to the conditions
		\begin{align}\label{eqn: QQQ}
		k_1\in [2/3,4],k_2\in [10/9,50/9],k_1/k_2\in [9/25,6/5].
		\end{align}
		The region determined by conditions (\ref{eqn: Q}),(\ref{eqn: QQQ}) is illustrated in Figure \ref{fig:figure 1}.
		\begin{figure}[h]
			\includegraphics[width=0.6\linewidth, height=0.6\linewidth]{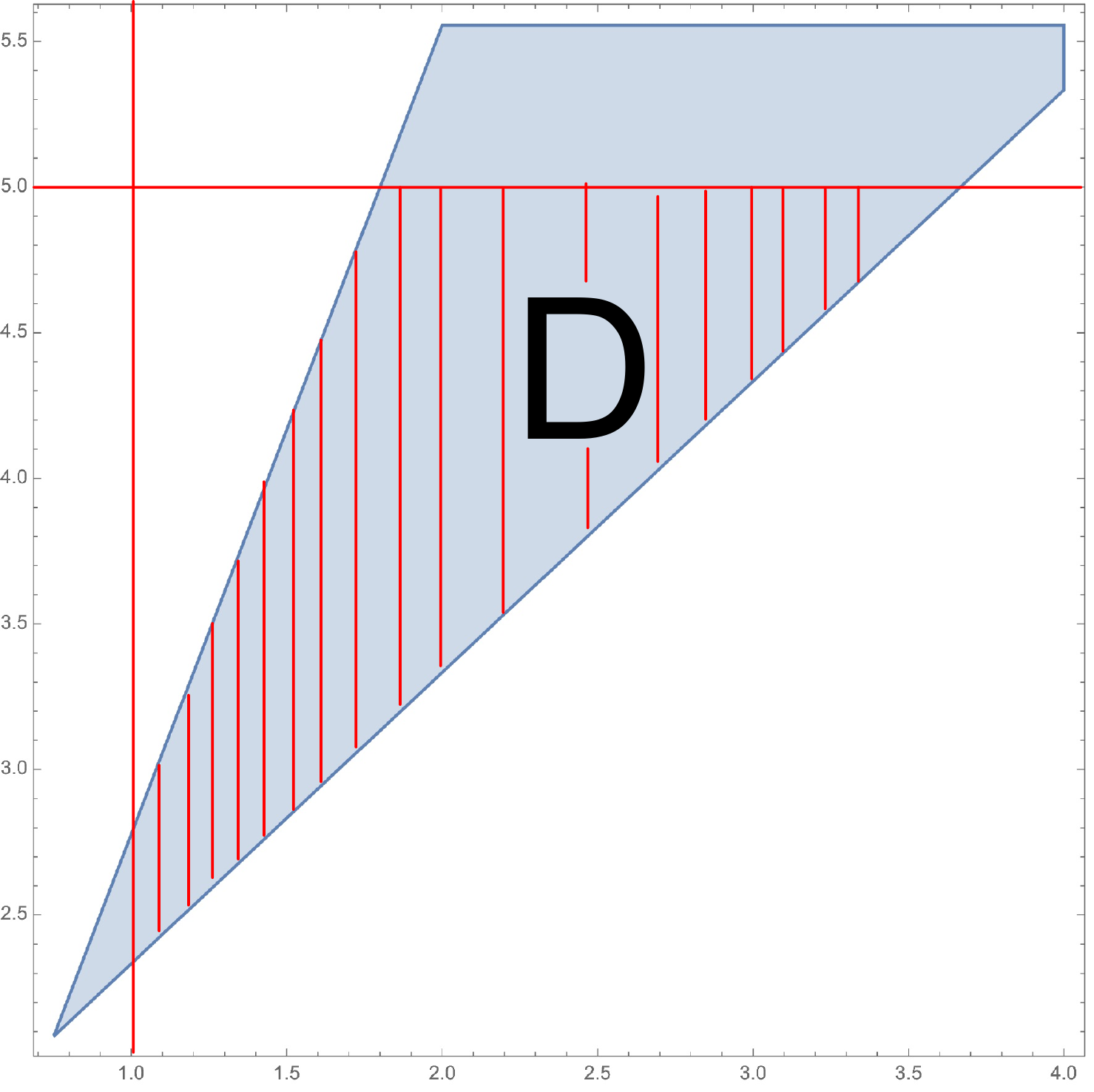} 
			\caption{Conditions (\ref{eqn: Q}),(\ref{eqn: QQQ}) (shaded region) and the set $D$ (covered with vertical lines)}
			\label{fig:figure 1}
		\end{figure}
		From  Figure \ref{fig:figure 1}, it is possible to see that $j_1=j_2=1$ will satisfy the claim. For later use, we write $D$ for the set of pairs $(k_1,k_2)$ with (\ref{eqn: Q}),(\ref{eqn: QQQ}) and $k_1\in [1,4],k_2\in [1,5].$ See the shaded region covered by vertical lines in Figure \ref{fig2}. Next, observe that the normal vector of $H_{1,1,k}$ is $$(1,4^{1-\{k\log 3/\log 4\}\}},-5^{1-\{k\log 3/\log 5\}\}}),$$ namely, in terms of $k_1,k_2$ above,
		\[
		k_1(k)=4^{1-\{k\log 3/\log 4\}\}}, k_2(k)=5^{1-\{k\log 3/\log 5\}\}}.
		\] 
		By Schanuel's conjecture, we see that the set
		\[
		\{(k\log 3/\log 4,k\log 3/\log 5)\mod\mathbb{Z}^2\}_{k\geq 1}\subset [0,1]^2
		\]
		is the orbit of an irrational rotation. Therefore, it is dense in $[0,1]^2.$ From here we see that there are infinitely many $k$ such that $(k_1(k),k_2(k))\in D.$ For such a $k,$ we have $H_{1,1,k}\cap A_3\times A_4\times A_5\neq\emptyset.$ This proves the claim with $j_1=j_2=1.$
		
		Now we know  that there are infinitely many  integers $k$ so that $H_{1,1,k}\cap A_3\times A_4\times A_5\neq \emptyset.$  Now $x',y',z'$ contains only $\{0,1\}$ in their $3,4,5$-ary expansions respectively. However, they might not be integers. If we take the integer parts we see that $[x'],[y'],[z']$ contains only $\{0,1\}$ in their $3,4,5$-ary expansions respectively and
		\[
		[x']+[y']=[z']+\{z'\}-\{x'\}-\{y'\}.
		\]
		The above equation tells us that
		\[
		\{z'\}-\{x'\}-\{y'\}
		\]
		is an integer. Observe that $\{x'\}, \{y'\}, \{z'\}$ are positive numbers whose $3,4,5$-ary expansions (respectively) contains only digits $0$ and $1$. Thus we see that
		\[
		\{x'\}\in (0,1/2], \{y'\}\in (0,1/3], \{z'\}\in (0,1/4].
		\] 
		This implies that
		\[
		-\frac{5}{6}<\{z'\}-\{x'\}-\{y'\}\leq \frac{1}{4}.
		\]
		The only integer in this range is $0.$ Thus, we see that
		\[
		\{z'\}-\{x'\}-\{y'\}=0
		\]
		and
		\[
		[x']+[y']=[z'].
		\]
		From here, the result follows from Schanuel's conjecture.
		
		Now, although we strongly do not believe it, it can be the case that $1$, $\log 3/\log 4$, $\log 3/\log 5$ are $\mathbb{Q}$-dependent. Then the rotation on $\mathbb{T}^2$ generated by the translation vector $(\log 3/\log 4, \log 3/\log 5)$ degenerates to an irrational rotation on a subtorus of dimension one. To be more precise, suppose that there are integers $k_1,k_2,k_3$ such that
		\[
		k_1\frac{\log 3}{\log 4}+k_2\frac{\log 3}{\log 5}+k_3=0.
		\]
		Neither $k_1$ nor $k_2$ is zero. For example, suppose that $k_1=0,$ then this implies that $1,\log 3,\log 5$ are not $\mathbb{Q}$-linearly independent, which is not possible. For convenience, we write $a=\log 3/\log 4, b=\log 3/\log 5.$  We see that
		\[
		k_1 a+k_2 b=-k_3.
		\] 
		We can find non-zero coprime integers $l_1,l_2$ such that
		\begin{align}\label{eqn: @@}
		l_1a+l_2b=c\in\mathbb{Q}.
		\end{align}
		Here $c\neq 0$ since otherwise $\log 5/\log 4=a/b\in\mathbb{Q}$ which is not the case. Since $\gcd (l_1,l_2)=1$ it is possible to find $S\in SL_2(\mathbb{Z})$ with entries 
		\[
		S=
		\begin{pmatrix}
		l_1 & l_2 \\
		l'_1 & l'_2
		\end{pmatrix}.
		\]
		Thus, $S$ is a well defined invertible map $\mathbb{T}^2\to\mathbb{T}^2.$ Let $k$ be an integer and consider the point $(ka,kb)\in\mathbb{T}^2.$ Now, we see that
		\begin{align*}
		S(ka,kb)&=&(l_1ka+l_2kb,l'_1ka+l'_2kb)\mod \mathbb{Z}^2\\
		&=& (kc,k(l'_1a+l'_2b))\mod \mathbb{Z}^2.
		\end{align*}
		It is simple to check that $l'_1a+l'_2b\notin\mathbb{Q}$ for otherwise both $a,b$ are rational, which is impossible. 
		
		We identify $\mathbb{T}^2$ with $[0,1]^2$ by letting
		\[
		(x,y)\mod \mathbb{Z}^2
		\] 
		be the unique point $(x',y')\in [0,1)^2$ such that $(x',y')-(x,y)\in\mathbb{Z}^2.$ Under this identification, the closure $\overline{\{S(ka,kb)\}_{k\geq 0}}$ is a union of vertical line segments in $[0,1]^2$.  In particular, it contains the vertical line $\{x=0\}\cap [0,1]^2.$ Performing $S^{-1},$ we see that $\overline{\{(ka,kb)\}_{k\geq 0}}$ contains the line
		\[
		L_0=S^{-1}(\{x=0\})=\{(x,y)\in\mathbb{T}^2: l_1x+l_2y=0\}.
		\]
		Although in general $L_0$ is a union of line segments in $[0,1]^2$, we will call it to be a line. In fact, it is actually a line in $\mathbb{T}^2.$ Since $c\neq 0,$ we write $c=p/q$ or $c=-p/q$ with positive  integers $p,q$ so that $\gcd(p,q)=1.$ The with similar arguments as above we see that  $\overline{\{(ka,kb)\}_{k\geq 0}}$ also contains
		\[
		L_{t}=\{(x,y)\in [0,1]^2: l_1x+l_2y\in\mathbb{Z}+t\}
		\]
		for $t\in\{1/q,2/q,\dots,(q-1)/q\}.$

		Next, we use our knowledge of the set $D$ in Figure \ref{fig:figure 1}.  We want to choose an integer $k$ such that
		\[
		V_k=(4^{-\{k\log 3/\log 4\}+1},5^{-\{k\log 3/\log 5\}+1})\in D.
		\]
		The closure of $V_k,k\geq 1$ is a union of curves. After performing the logarithmic map $\mathrm{Log}:(x,y)\to (1-\log x/\log 4,1-\log y/\log 5)$ those curves become line segments. As we discussed in above, $\overline{\{V_k\}_{k\geq 1}}$ contains the line $L_0.$ 
		
		Recall that neither $l_1$ nor $l_2$ is zero. To find $V_k\in D,$ it is enough to find $k$ so that
		\[
		(\{k\log 3/\log 4\},\{k\log 3/\log 5\})\in \mathrm{Log}(D).
		\]
		To do this, it is enough to show that at least one of the lines \[L_t,t\in\{0,1/q,2/q,\dots,(q-1)/q\}\] intersect the interior of  $\Log(D).$ If this is done, then we can find infinitely many $k$ so that $V_k\in D.$ The region $\Log(D)$ is illustrated in Figure \ref{fig2}.
		\begin{figure}[h]
			\includegraphics[width=0.5\linewidth, height=0.5\linewidth]{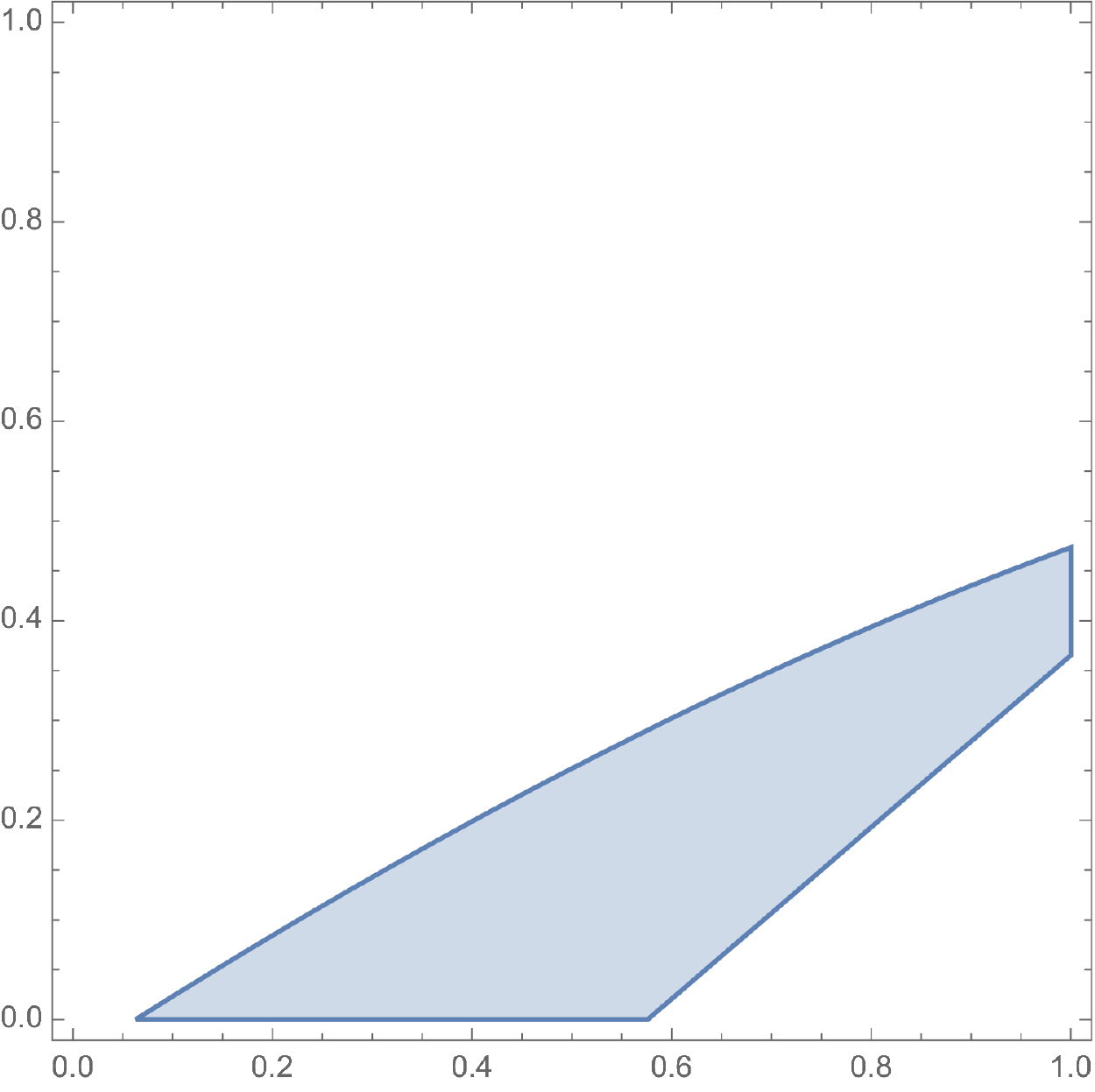} 
			\caption{The region $\Log(D)$}
			\label{fig2}
		\end{figure}
		
		In order that the line $L_0$ avoids $\Log(D)$, both $l_1,l_2$ cannot be too large. First, $\Log(D)\cap \{y=0\}$ has length bigger than $0.5.$ Since $L_0$ contains points $(0,1),(1/|l_1|,1),\dots,((|l_1|-1)/|l_1|,1),(1,1),$ we see that $|l_1|<2.$ Therefore $l_1=\pm 1.$ Without loss of generality, we can assume that $l_1=1.$ Then we can consider possible values for $l_2.$ We see that $l_2<0$. Otherwise if $l_2>0,$ $L_0$ must contain a line segment starting from $(1,0)$ to the boundary $\{x=0\}$ and this line segment must intersect the interior of  $\Log(D)$ unless it is horizontal which is not possible.  Similarly, as long as $l_2\leq -3,$ $L_0$ contains a line segment starting from a point $(1,r), r\in[0,1/3]$ to the boundary $\{x=0\}$ and this line segment must intersect the interior of  $\Log(D).$ Thus we see that all possible values for $l_2$ are $-1,-2.$ 
		
		Now, we consider all possible values for $c=p/q.$ We see that $(t,0)\in  L_t$ for $t\in\{1/q.\dots,(q-1)/q\}.$ Then with a similar argument as above we see that $q< 2.$ Therefore $q=1.$ This implies that $c\in\mathbb{Z}.$ 
		
		In conclusion, in order that  $V_k\in D$ for at most finitely many integers $k,$ we must have
		\[
		l_1a+l_2b\in\mathbb{Z}
		\]
		for $l_1=1,l_2=-1 \text{ or } -2.$ It can be checked directly that
		\[
		\frac{\log 3}{\log 4}-\frac{\log 3}{\log 5}=0.109875+e_1,
		\]
		\[
		\frac{\log 3}{\log 4}-2\frac{\log 3}{\log 5}=-0.572731+e_2,
		\]
		where $|e_1|,|e_2|<0.001.$ Thus, the numbers on the LHS above are not integers. From here, the proof is finished.

	\end{proof}
	We discuss more on Theorem \ref{EGRSIII}. First, we remark that with the same proof, one can also show that the conclusion holds with $(3,4,5)$ being replaced with $(3,4,7).$ More generally, let $l\geq 1$ be an integer. Let $D=\{0,\dots,l\}.$ We want to consider the set ($N^D_{b_1}+N^D_{b_2})\cap N^D_{b_3}$ for $b_1,b_2,b_3\geq l+1.$ In this case, under Schanuel's conjecture, it is possible to show that
	$$(N^D_{b_1}+N^D_{b_2})\cap N^D_{b_3}$$
	in infinite as long as
	\[
	\frac{l}{b_1-1}+\frac{l}{b_2-1}+\frac{l}{b_3-1}\geq 1,
	\]
	and $1,\log b_1,\log b_2,\log b_3$ are $\mathbb{Q}$-linearly independent. Again, for specific cases, Schanuel's conjecture may not be required. For example, when $l=2,$ it is possible to check that the above statement holds for $(b_1,b_2,b_3)=(4,5,6).$
	
	Lastly, we are going to prove Theorem \ref{EGRS4}. Before that, we first show the following results.
	\begin{lma}\label{temp lma}
	Let $b\geq 3$ be an integer and let $B=\{1,2,\dots,b-1\}.$ Consider the set
	\[
	A^{B}_b=\{x\in [0,1]: \text{ some $b$-ary expansion of $x$ contains only digits in $B$} \}.
	\]
	Then we have $S(A^B_b)=(b-2)/(b-1).$ Moreover, let $k\in\{0,1,\dots,b-2\}$. Consider the set
	\[
	\tilde{A}^B_b=A^B_b\cap [0,(b-k)/b].
	\]
	Then $S(\tilde{A}^B_b)=S(A^B_b).$
	\end{lma}
	\begin{rem}
	The second conclusion is special for our self-similar set $A^B_b.$ We need to cut $A^B_b$ at the 'right places'. For general Cantor sets, cutting out a small portion may decrease the thickness dramatically. 
	\end{rem}
	\begin{proof}
	    The first conclusion follows from Lemma \ref{THICK}. Indeed, by applying the symmetry $x\in\mathbb{R}\to -x\in\mathbb{R}$ to the set $A^B_b,$ we obtain a set for which Lemma \ref{THICK} can be used. 
	    
	    For the second conclusion, we can use the self-similarity of $A^B_b.$ First, let $[a,1]$ be the convex hull of $A^B_b$ (this determines the value for $a\in (0,1/b)$). Then we see that $[a,(b-k)/b]$ is the convex hull of $\tilde{A}^B_b$. Observe that $A^B_b\cap [1/b,2/b],\dots,A^B_b\cap [(b-k-1)/b,(b-k)/b]$ are translated copies of each other. The largest gaps of $\tilde{A}^B_b$ are located inside each of those copies. From here we see that
	    \[
	    S(\tilde{A}^B_b)=(b-2)/(b-1).
	    \]
	\end{proof}
	\begin{lma}\label{tmp lma 2}
	Let $\delta_1,\delta_2\in (0,1)$ numbers so that $\delta_1<\delta_2.$  Let $C\subset [0,1]$ be a Cantor set with $Conv(C)\neq [0,1],$ $S(C)>1-\delta_1$ and $|Conv(C)|>1-\delta_1.$ Consider the set $A_b^{B}$ as in Lemma \ref{temp lma}. Then for each $\delta_2\in (0,1),$ as long as $\delta_1$ is small enough, for all sufficiently large $b>1,$ there is a compact set $C'\subset C\cap A_b^{B}$ satisfying
	\[
	|Conv(C')|>1-\delta_2,
	\]
	\[
	S(C')>1-\delta_2.
	\]
	Moreover, there is a number $c>0$ such that for linear maps $T:x\to rx+t, T':x\to r'x+t'$ with $r,r'\in (1-c,1+c)$ and $t,t'\in (-c,c),$ there is a compact set $C'\subset T(C)\cap T'(A_b^{B})\cap [0,1]$ satisfying the above properties as well.
	\end{lma}
	\begin{proof}
	    We want to use Theorem \ref{NEWHOUSEMULTI}. The issue is that $Conv(C)\cap Conv(A_b^B)$ may contain $C$ or $A_b^B.$ Observe that $Conv(A_b^B)=[a_b,1]$ with $\lim_{b\to \infty} a_b\to 0$, and that $\lim_{b\to\infty} S(A_b^B)\to 1.$ So as long as $b$ is large enough, $C, A_b^B$ are not contained in each other's gaps. Let $Conv(C)=[c,d]\subset [0,1].$ 
	    
	    If $c=0,$ then $d<1$ and we see that $Conv(C)\cap Conv(A_b^B)$ contains neither $C$ nor $A_b^B.$ Using Theorem \ref{NEWHOUSEMULTI}, we conclude that as long as $\delta_1$ is small enough and $b$ is large enough, there is a compact set $C'\subset C\cap A^B_b$ such that
	    \[
	    S(C'), |Conv(C')|
	    \]
	    are both greater than $1-\delta_2.$
	    
	    If $c>0,$ by making $b$ sufficiently large, we can then assume that $a_b<c.$ As $|Conv(C)|>1-\delta_1$ we see that $d> 1-\delta_1.$ Let $k$ be the smallest integer with $(b-k)/b< d.$ Consider the set $\tilde{A}_b^B$ as in Lemma \ref{temp lma}. We see that
	    \[
	    |Conv(\tilde{A}_b^B)|=\frac{b-k}{b}-a_b>1-\delta_1-\frac{1}{b}-a_b.
	    \]
	    This number can be made arbitrarily close to $1-\delta_1$ by making $b$ sufficiently large (in a manner that depends on $\delta_1$). Now since $Conv(C)\cap Conv(\tilde{A}_b^B)$ contains neither $C$ nor $\tilde{A}_b^B,$ by Theorem \ref{NEWHOUSEMULTI}, we see that (for small enough $\delta_1$ and large enough $b$) there is a compact set $C'\subset C\cap \tilde{A}^B_b\subset C\cap A^B_b$ such that
	    \[
	    S(C'), |Conv(C')|
	    \]
	    are both greater than $1-\delta_2.$ 
	    
	    For the second conclusion, notice that  if $Conv(C)\cap Conv(A_b^B)$ does not contain $C, A_b^B,$ then as long as $c$ is small enough, $Conv(T(C))\cap Conv(T'(A_b^B))$ does not contain $T(C), T'(A_b^B).$ The same conclusion holds with $A_b^B$ being replaced with $\tilde{A}_b^B$. Thus it is possible to find a compact subset $C'\subset T(C)\cap T'(A_b^B)$ satisfying the desired properties. However, later on, it is more convenient to consider $T(C)\cap T'(A_b^B)\cap [0,1].$ We cannot simply take $C'\cap [0,1]$ as the thickness can decrease significantly. Instead, we consider the set $T'(A_b^B)\cap [0,1].$ The set $T'(A_b^B)$ is a self-similar set. The scaling ratio is still $1/b.$ The first level branches are intervals of length $r'/b.$ It can happen that some of the first level branches are not contained in $[0,1].$ If this is the case, we simply ignore them. We can choose $c$ to be so small that we only need to ignore at most one first level branch. As a result, as long as $c$ is small enough, we obtain a subset $\tilde{\tilde{A}}_b^B\subset T'(A_b^B)\cap [0,1]$ whose thickness is the same as $A_b^B$ and whose diameter is at least $1-3r'/b$. We can then consider the intersection between $T(C)$ and $\tilde{\tilde{A}}_b^B$ as in the case when $T,T'$ are the identity map. From here the result follows.
	\end{proof}

	\begin{proof}[Proof of Theorem \ref{EGRS4}] Let $3\leq b_1\leq\dots\leq b_k$ be (not necessarily distinct) integers. We need them to be sufficiently large in a manner that will be discussed later in the proof.
		
		For each $i\in\{1,\dots,k\}$ we construct the set
		\[
		A_i=\{x\in (0,\infty): \text{some $b_i$-ary expansion of $x$ does not contain digit $0$}\}.
		\]
		If $x$ has two possible $b_i$-ary expansions, we will take the finite expansion. Thus the $b_i$-ary expansion for $x>0$ is well defined.  It is sufficient to consider the set
		\[
		A=A_1\cap A_2\dots \cap A_k.
		\]
		We claim that it is enough to prove that $A$ is unbounded. Indeed, if $A$ is unbounded, then we can find arbitrarily large $x$ whose base $b_1,\dots,b_k$ expansions do not have digit zero.  We simply take the integer part of $x$ to obtain an integer $[x]$ whose base $b_1,\dots,b_k$ expansions do not have digit zero. Since $A$ is unbounded, the proof finishes.
		
		Now it is enough to prove that $A$ is unbounded. We define the integer sequence
		\[
		\mathbf{a}_j=(a_{1,j},\dots,a_{k,j}),j\geq 0
		\]
		by starting with
		\[
		\mathbf{a}_0=(1,\dots,1).
		\]
		For $j\geq 1,$ define $a_{1,j}=b_1a_{1,j-1}.$ For $i\in\{2,\dots,k\},$ we define $a_{i,j}=b_ia_{i,j-1}$ if $a_{i,j-1}/a_{1,j}<b^{-1}_i$ and $a_{i,j}=a_{i,j-1}$ otherwise. In this way, we have
		\[
		\mathbf{a}_j=(b^j_1,b^{[j\log_{b_2} b_1]}_2,\dots,b^{[j\log_{b_k} b_1]}_k)=b^j_1(1,b^{-\{j\log _{b_2} b_1\}}_2,\dots,b^{-\{j\log _{b_k} b_1\}}_k).
		\]
		The rotation on $\mathbb{T}^{k-1}$ with the rotation angle
		\[
		v=\left(\frac{\log b_1}{\log b_2},\dots,\frac{\log b_1}{\log b_k}\right)
		\]
		may not be an irrational rotation on $\mathbb{T}^{k-1}$.  Nonetheless, we claim that the orbit \[nv\mod\mathbb{Z}^{k-1},n\geq 0\] can be close to the origin, i.e. $d(nv,\mathbb{Z}^{k-1})$ can be arbitrarily small. In fact, $\overline{\{nv\mod\mathbb{Z}^{k-1}\}_{k\geq 0}}$ is a set of form ${T'}+P$ where $T'$ is a (possibly trivial) subtorus, and $P$ is a finite set of rational points containing $(0,\dots,0)$. It is possible that $T'$ is a singleton in which case $nv\mod\mathbb{Z}^{k-1}$ is periodic. In any case, $T'$ contains the origin and this proves the claim. 
		
		Let $\delta'\in (0,1/2).$ Let $\epsilon>0$ be a small number such that
		\begin{align}\label{delta}
		b^{-\epsilon}_i>1-\delta',i\in\{1,2,\dots,k\}.
		\end{align}
	Notice that there are infinitely many $n$ such that $d(nv,\mathbb{Z}^{k-1})<\epsilon.$ Let $n$ be such an integer. Then we see that
		\[
		\mathbf{a}_n=b^n_1(1,v_2,v_3,\dots,v_k),
		\]
		where $v_i\in (1-\delta',1].$ For each $i\in\{1,2,\dots,k\}$ let 
		\[B_{i,n}=A_i\cap [a_{i,j},b_{i} a_{i,j}].\] 
		We claim that if $b_1,b_2,\dots,b_k$ are  sufficiently large, then $B_{1,n}\cap B_{2,n}\cap \dots \cap B_{k,n}\neq \emptyset.$ The most convenient way to look at this is to rescale the whole situation by a factor of $b^{-n}_1.$ After doing this, $B_{1,n}$  fits to the interval $[1,b_1]$ while $B_{2,n},\dots, B_{k,n}$ fit to the intervals $[v_2,b_2v_2],\dots, [v_k,b_kv_k].$ As one can choose $\delta'$ to be arbitrarily small, $v_1,v_2,\dots,v_k$ can be arbitrarily close to one. It is enough to consider the situation with $1=v_1=v_2=v_3=\dots=v_k.$ Indeed, both the thickness and the gap conditions in Theorems \ref{NEWHOUSE}, \ref{NEWHOUSEMULTI} are preserved if one rescales and translates each Cantor set slightly. We will come back to this point shortly. If $v_1,v_2,\dots, v_k$ are exactly $1,$ we see that it is enough to consider the intersection
		\[
		C_1\cap C_2\dots\cap C_k,
		\]
		where $C_i=A_i\cap [1,b_i].$ The idea is to use Theorem \ref{NEWHOUSEMULTI} and Lemma \ref{tmp lma 2} inductively. We find large (in diameter) and thick (in normalized thickness) Cantor sets contained in $C_1\cap C_2, C_1\cap C_2\cap C_3,\dots, C_1\cap\dots\cap C_k.$ After that, we show that this procedure still goes through if one changes $C_1,\dots,C_k$ by applying affine maps. This helps us to restore the information when $v_1,\dots,v_k$ are almost but not exactly one.
		
		Now, $C_i$ is a self-similar set with contraction ratio $1/b_i.$ Indeed, 
		\[
		C_i\cap [1,2], C_i\cap [2,3],\dots, C_i\cap [b_i-1,b_i]
		\]
		are all scaled copies of $C_i.$ The largest bounded gaps of $C_i$ are of length $1/(b_i-1)$ and $S(C_i)=(b_i-2)/(b_i-1).$ 
		
		We now start to show the base step for the inductive argument. It is simple to check that $C_1\cap [1,2],C_2\cap [1,2]$ satisfy the gap condition of Theorem \ref{NEWHOUSEMULTI}. Thus for each $\delta>0,$ we can find a compact subset $C_{1,2}\subset C_1\cap C_2\cap [1,2]$ with $S(C_{1,2})>1-\delta$ as long as $b_1,b_2$ are sufficiently large. 
		
		We now show that the diameter of $C_{1,2}$ can be arbitrarily close to one at the cost of making $b_1,b_2$ sufficiently large. By Lemma \ref{tmp lma 2}, we see that as long as $b_1,b_2$ are large enough, there is a compact set $C_{1,2}\subset C_1\cap C_2\cap [1,2]$ with thickness and diameter both at least $1-\delta.$ 
		
		We can apply the same argument inside the intervals $[2,3],[3,4],\dots,[b_1-1,b_1].$ We now replace $C_{1,2}$ with the union of all such compact sets (in $[1,2], [2,3],\dots$).
		
		To summarize, for each $\delta>0,$ as long as $b_1,b_2$ are large enough, we can find a compact set $C_{1,2}\subset C_1\cap C_2$ so that for each $l\in\{1,\dots,b_1-1\}$, $S(C_{1,2}\cap [l,l+1]), |Conv(C_{1,2}\cap [l,l+1])|$ are at least $1-\delta.$
		From the second part of Lemma \ref{tmp lma 2}, we conclude that for each $\delta>0,$ as long as $b_1,b_2$ are large enough, there is a number $c>0$ such that for all affine maps $T_1,T_2$ with scaling ratios $c$-close to one and translations $c$ close to zero, there is a compact set $C_{1,2}\subset T_1(C_1)\cap T_2(C_2)$ such that for each $l\in \{1,\dots,b_1-1\}$, $S(C_{1,2}\cap [l,l+1])$ and $|Conv(C_{1,2}\cap [l,l+1])|$ are both larger than $1-\delta$. Moreover, it is possible to see that $Conv(C_{1,2}\cap [l,l+1])\neq [l,l+1].$ This is because if $c$ is small enough, then $Conv(T_1(C_1)\cap [l,l+1])\neq [l,l+1], Conv(T_2(C_2)\cap [l,l+1])\neq [l,l+1]$. This finishes the base step.
		
		We now perform the induction. Let $j\in\{2,3\dots,k-1\}$. Suppose that for each $\delta>0,$ as long as $b_1,\dots,b_{j}$ are large enough, there is a number $c>0$ such that as long as $r_1,\dots,r_j\in (1-c,1+c)$, $t_1,\dots,t_j\in (-c,c)$ we can find a compact set
		\[
		C_{1,2,\dots,j}\subset \bigcap_{i=1}^j T_i(C_i)
		\]
		such that for each $l\in\{1,\dots,b_1-1\}$ we have that $Conv(C_{1,2\dots,j}\cap [l,l+1])\neq [l,l+1]$, and that
		\[
		S(C_{1,2,\dots,j}\cap [l,l+1]), |Conv(C_{1,2,\dots,j}\cap [l,l+1])|
		\]
		are greater than $1-\delta$ where $T_i,i\in\{1,\dots,j\}$ are linear maps $x\to r_ix+t_1.$ We call this statement to be $ST(j).$ Under $ST(j),$ inside each interval $[l,l+1],l\in\{1,\dots,b_1-1\}$, it is possible to use Lemma \ref{tmp lma 2} for $C_{1,2\dots,j}$ and $C_{j+1}$. From here we see that $ST(j+1)$ holds.

		 Finally, by induction, we see that $ST(k)$ holds. Thus there is a small number $c>0$ such that as long as $r_1,r_2,\dots,r_k\in (1-c,1+c)$ and $t_1,t_2,\dots,t_k\in (-c,c)$ we can find a compact set
		\[
		C_{1,2,\dots,k}\subset \bigcap_{i=1}^k T_i(C_i)
		\]
		 such that for each $l\in\{1,\dots,b_1-1\},$ $C_{1,2,\dots,k}\cap [l,l+1]$ has thickness at least $1-\delta$ and diameter at least $1-\delta,$ where $T_i:x\to r_ix+t_i$. In particular, we see that $\bigcap_{i=1}^k T_i(C_i)\neq\emptyset.$ Now we can choose $\delta'$ in (\ref{delta}) to be small enough according to $c$. We can then rescale this situation by $b^n_1$ and conclude that
		\[
		B_{1,n}\cap B_{2,n}\cap\dots\cap B_{k,n}\neq\emptyset.
		\]
		We proved that for sufficiently large $b_1,\dots,b_k,$ there are infinitely many $n\geq 0$ such that
		\[
		B_{1,n}\cap B_{2,n}\cap \dots\cap B_{k,n}\neq \emptyset.
		\]
		This implies that $A=A_1\cap A_2\cap \dots\cap A_k$ is unbounded, and the proof of the theorem is finished.
	\end{proof}
	
	\section{Acknowledgement}
	HY was financially supported by the University of Cambridge and the Corpus Christi College, Cambridge. HY has received funding from the European Research Council (ERC) under the European Union's Horizon 2020 research and innovation programme (grant agreement No. 803711). HY especially thank an anonymous referee for many instructive comments, which helped to improve the manuscript significantly.

	\bibliographystyle{amsplain}
	
\end{document}